\documentclass[10pt,letterpaper]{article}
\usepackage[utf8]{inputenc}
\usepackage{amsmath}
\usepackage{amsfonts}
\usepackage{amssymb}
\usepackage{graphicx}
\usepackage{mathrsfs}
\usepackage{upref,amsthm,amsxtra,exscale}
\usepackage{cite}
\usepackage[colorlinks=true,urlcolor=blue,
citecolor=red,linkcolor=blue,linktocpage,pdfpagelabels,
bookmarksnumbered,bookmarksopen]{hyperref}
\usepackage{upgreek}
\usepackage{cleveref}

\usepackage{fullpage}

\newtheorem{theorem}{Theorem}[section]

\newtheorem{lemma}[theorem]{Lemma}

\newtheorem{problem}[theorem]{Problem}
\newtheorem{remark}[theorem]{Remark}
\numberwithin{equation}{section}

\def\r{\mathbb{R}}
\def\s{\mathbb{S}}
\def\rn{\mathbb{R}^N}

\def\eps{\varepsilon}
\def\rh{\rightharpoonup}
\def\io{\int_{\Omega}}
\def\irn{\int_{\r^N}}
\def\vp{\varphi}
\def\vr{\varrho}
\def\o{\Omega}

\def\cC{\mathcal{C}}

\def\cE{\mathcal{E }}

\def\cN{\mathcal{N}}
\def\cO{\mathcal{O}}

\def\cU{\mathcal{U}}

\def\cO{\mathcal{O}}
\def\d{\,\mathrm{d}}

\def\gen{\mathrm{genus}}

\def\what{\widehat}

\author{Mónica Clapp \and Víctor Hernández-Santamaría\thanks{The work of V. Hern\'andez-Santamar\'ia is supported by the program ``Estancias Posdoctorales por México para la Formación y Consolidación de las y los Investigadores por México'' of CONAHCYT (Mexico). He also received support from CONAHCYT grant CBF2023-2024-116 and from UNAM-DGAPA-PAPIIT grants IN109522, IA100324, and IA100923 (Mexico).} \and Alberto Saldaña\footnote{ A. Saldaña is supported  by
CONAHCYT grants CBF2023-2024-116 (Mexico) and by UNAM-DGAPA-PAPIIT grant IA100923 (Mexico).
}}
\title{Positive and nodal limiting profiles for a semilinear elliptic equation with a shrinking region of attraction}
\date{}

\begin{document}
\maketitle
	
\begin{abstract}
We study the existence and concentration of positive and nodal solutions to a Schrödinger equation in the presence of a shrinking self-focusing core of arbitrary shape. Via a suitable rescaling, the concentration gives rise to a limiting profile that solves a nonautonomous elliptic semilinear equation with a sharp sign change in the nonlinearity.  We characterize the (radial or foliated Schwarz) symmetries and the (polynomial) decay of the least-energy positive and nodal limiting profiles.

\medskip

\medskip

\end{abstract}

\section{Introduction}

Let $\o$ be a bounded open subset of $\rn$, not necessarily connected, $N\geq 3$, and $Q:\rn\to\r$ be the function
\begin{equation*}
Q(x):=
\begin{cases}
1 & \text{if \ }x\in\o, \\
-1 & \text{if \ }x\in\rn\smallsetminus\o.
\end{cases}
\end{equation*}
Consider the problem
\begin{equation} \label{eq:Q_eps_problem}
\begin{cases}
-\Delta v+v=Q_\eps(x)|v|^{p-2}v,\\
v\in H^1(\rn),
\end{cases} 
\end{equation}
where $\eps>0$, $Q_\eps(x):=Q(\frac{x}{\eps})$, i.e., $Q_\eps(x)=1$ if $x\in\eps\o$ and  $Q_\eps(x)=-1$ if $x\in\rn\smallsetminus\eps\o$, $p\in (2,2^*)$ and $2^*$ is the critical Sobolev exponent, namely, $2^*:=\frac{2N}{N-2}$.

Equations of this kind occur in some models of optical waveguides propagating through a stratified dielectric medium, see \cite{st1,st2}. The amplitude of the electric field is given by a positive ground state. The nonlinear term of~\eqref{eq:Q_eps_problem} is related to the nonlinear contribution of the dielectric response. The response is said to be defocusing at $x$ if $Q_\eps(x)<0$ and it is called self-focusing at $x$ if $Q_\eps(x)>0$. So the problem~\eqref{eq:Q_eps_problem} for small $\eps$ describes the situation where the medium has a self-focusing core. A detailed discussion may be found in \cite{as}.

Ackermann and Szulkin studied this problem in \cite{as} and showed that the positive least energy solutions exhibit concentration as $\eps\to 0$. When $\o$ is the unit ball, Fang and Wang proved in \cite{fw} that the limit profile of these solutions is a least energy solution to the problem
\begin{equation} \label{eq:limit_problem}
\begin{cases}
-\Delta w=Q(x)|w|^{p-2}w,\\
w\in E,
\end{cases} 
\end{equation}
where $E:=D^{1,2}(\rn)\cap L^p(\rn)$. See also \cite{css} where competitive Schrödinger systems with shrinking regions of attraction are considered, and \eqref{eq:limit_problem} also appears as a limit problem in some cases. Thus, it is of interest to obtain information on the qualitative properties of the solutions to~\eqref{eq:limit_problem}. This is one of our goals.

We start by showing that, for each $\eps>0$, the problem~\eqref{eq:Q_eps_problem} has infinitely many solutions, that it has a positive least energy solution and a least energy nodal solution and that both of them exhibit concentration as $\eps\to 0$. More precisely, we prove the following results.

\begin{theorem} \label{thm:main_multiplicity}
For each $\eps>0$ the problem~\eqref{eq:Q_eps_problem} has infinitely many solutions.
\end{theorem} 

The following result is an easy extension of \cite[Theorem 1.1]{fw} to an arbitrary domain $\o$.

\begin{theorem} \label{thm:main_positive}
Let $\eps_n\to 0$. For each $n$ the problem~\eqref{eq:Q_eps_problem} with $\eps=\eps_n$ has a positive least energy solution $v_n$. Set $u_n(x):=\eps_n^\frac{2}{p-2}v_n(\eps_nx)$. Then, after passing to a subsequence, $(u_n)$ converges strongly in $E$ to a positive least energy solution of the limit problem~\eqref{eq:limit_problem}. As a consequence, for any $\vr>0$,
\begin{equation} \label{eq:concentration}
\lim_{n\to\infty}\frac{\int_{|x|\leq\vr}(|\nabla v_n|^2+v_n^2)}{\irn(|\nabla v_n|^2+v_n^2)}=1\qquad\text{and}\qquad\lim_{n\to\infty}\frac{\int_{|x|\leq\vr}|v_n|^p}{\irn |v_n|^p}=1.
\end{equation}
\end{theorem}

We obtain a similar result for sign-changing solutions.

\begin{theorem} \label{thm:main_nodal}
Let $\eps_n\to 0$. For each $n$ the problem~\eqref{eq:Q_eps_problem} with $\eps=\eps_n$ has a least energy nodal solution $v_n$. Set $u_n(x):=\eps_n^\frac{2}{p-2}v_n(\eps_nx)$. Then, after passing to a subsequence, $(u_n)$ converges strongly in $E$ to a least energy nodal solution of the limit problem~\eqref{eq:limit_problem}. As a consequence, for any $\vr>0$, $(v_n)$ satisfies~\eqref{eq:concentration}.
\end{theorem}

Theorems~\ref{thm:main_positive} and~\ref{thm:main_nodal} state, in particular, that the limit problem~\eqref{eq:limit_problem} has a positive and a nodal least energy solution. Note that this is not true if we replace $Q$ by a constant function. Actually, the problem
$$-\Delta w=\kappa|w|^{p-2}w,\qquad w\in E.$$
has only the trivial solution for any $\kappa\in\r$. This is obvious if $\kappa\leq 0$ and it is a consequence of the Pohozhaev identity \cite[Theorem B.3]{w} if $\kappa>0$.

The proof of Theorems \ref{thm:main_multiplicity}, \ref{thm:main_positive}, and \ref{thm:main_nodal} is via variational methods. The fact that $Q=-1$ in $\rn\smallsetminus \Omega$ is crucial to guarantee compactness.

Now we turn our attention to understanding the shape and qualitative properties of the limiting profiles. A first remark is that standard regularity arguments show that every solution of the limit problem \eqref{eq:limit_problem} belongs to $W^{2,s}_{loc}(\rn)\cap \cC^{1,\alpha}_{loc}(\rn)\cap \cC_{loc}^\infty(\Omega)\cap \cC_{loc}^\infty(\rn\smallsetminus \overline{\Omega})$ for all $s\in[1,\infty)$ and $\alpha\in(0,1)$ (see Lemma~\ref{lem:regularity_limit} below). 

Our next result gives information on the symmetries of the positive and the nodal least energy solutions of \eqref{eq:limit_problem} when $\o$ is a radially symmetric open bounded set.

\begin{theorem} \label{thm:symmetries}
\begin{itemize}
\item[$(i)$] Let $\o$ be the unit ball in $\rn$ centered at the origin. Then, every positive least energy solution of~\eqref{eq:limit_problem} is radially symmetric and strictly decreasing in the radial direction.
\item[$(ii)$] Let $\o$ be a radially symmetric open bounded subset of $\rn$ with smooth boundary. Then, any least energy solution and any least energy nodal solution of \eqref{eq:limit_problem} is foliated Schwarz symmetric in $\rn$.
\end{itemize}
\end{theorem}

Note that Theorem \ref{thm:symmetries} $(ii)$ does not require that $\Omega$ is connected. So $\Omega$ may be, for instance, the disjoint union of a finite number of annuli centered at the origin. Theorem \ref{thm:symmetries} $(ii)$ says that any positive or nodal least energy solution is invariant under every rotation around some fixed axis in $\rn$ and that it is nonincreasing with respect to the polar angle; see Subsection \ref{fss:sec} for the precise definition of foliated Schwarz symmetry.

In the proof of Theorem~\ref{thm:symmetries} we use rearrangements (symmetrizations and polarizations).  Here, the main obstacle is that the nonlinearity changes sign (and, therefore, its monotonicity). However, we show that the influence of the equation in $\o$ is the predominant one to establish the symmetry in the whole of $\rn.$

Finally, we study the decay at infinity of solutions to~\eqref{eq:limit_problem}. We write $B_\vr$ for the ball of radius $\vr$ in $\rn$ centered at the origin.

\begin{theorem} \label{thm:decay}
\begin{itemize}
\item[$(i)$] Let $w$ be a solution of~\eqref{eq:limit_problem}. Then there exists $C>0$ (depending on $w$) such that
$$|w(x)|\leq C|x|^{2-N}\qquad\text{for every \ }x\in\rn.$$
\item[$(ii)$] Assume that $p\in(\frac{2N-2}{N-2},\frac{2N}{N-2})$ and let $\vr>0$ be such that $\o\subset B_\vr$. If $w$ is a positive solution of~\eqref{eq:limit_problem} then, for any given $\delta>0$, there exists $C_\delta>0$ (depending on $w$ and $\delta$) such that
$$w(x)\geq C_\delta|x|^{2-N-\delta}\qquad\text{for every \ }x\in\rn\smallsetminus B_\vr.$$
\end{itemize}
\end{theorem}
The number $\frac{2N-2}{N-2}$ is sometimes called the Serrin exponent, and it plays a role in the analysis of nonlinear problems with isolated singularities, see for instance \cite{v}.  The proof of Theorem~\ref{thm:decay} is carried out by building suitable sub and supersolutions, first in the radially symmetric setting, and then by extending this bounds by comparison to the general case.  It is interesting to remark that solutions of~\eqref{eq:Q_eps_problem} have exponential decay (see Remark~\ref{exp:rmk}), but the limit profiles have polynomial decay. We note that one cannot expect  statement $(ii)$ in Theorem~\ref{thm:decay} to be always true for small values of $p$. It is certainly not if $N=3$, see Remark~\ref{rem:lower bound}. The following question remains open.

\begin{problem}
Establish the precise decay of the solutions to problem~\eqref{eq:limit_problem} for every $p\in(2,2^*)$.
\end{problem}

The paper is organized as follows.  In Section~\ref{e:sec} we prove Theorem~\ref{thm:main_multiplicity} and the existence of a positive least energy solution of~\eqref{eq:Q_eps_problem}. Its limit profile is characterized in Section~\ref{poslim:sec}, where we prove Theorem~\ref{thm:main_positive}.  The existence of a least energy nodal solution of~\eqref{eq:Q_eps_problem} is studied in Section~\ref{n:sec} and its limit profile is characterized in Section~\ref{nodlimit:sec}. The proof of Theorem~\ref{thm:main_nodal} is given in that same section.  Finally, the proof of the decay estimates in Theorem~\ref{thm:decay} is given in Section~\ref{d:sec} and the proof of the symmetry results stated in Theorem~\ref{thm:symmetries} can be found in Section~\ref{s:sec}.

\section{Existence of positive least energy solutions}\label{e:sec}

Fix $\eps>0$. Setting $u(x):=\eps^\frac{2}{p-2}v(\eps x)$, the problem~\eqref{eq:Q_eps_problem} turns out to be equivalent to
\begin{equation} \label{eq:Q_problem}
\begin{cases}
-\Delta u+\eps^2 u=Q(x)|u|^{p-2}u,\\
u \in H^1(\rn),
\end{cases} 
\end{equation}
that is, $u$ is a solution of~\eqref{eq:Q_problem} if and only if $v$ is a solution of~\eqref{eq:Q_eps_problem}. Let
\begin{equation}\label{eq:eps_inner product}
\langle u,v\rangle_\eps:=\irn(\nabla u\cdot\nabla v + \eps^2 uv)\qquad\text{and}\qquad \|u\|_\eps^2:=\irn(|\nabla u|^2 + \eps^2u^2).
\end{equation}
The solutions of~\eqref{eq:Q_problem} are the critical points of the functional $J_\eps:H^1(\rn)\to\r$ given by
$$J_\eps(u):=\frac{1}{2}\irn(|\nabla u|^2 + \eps^2u^2) - \frac{1}{p}\irn Q(x)|u|^p,$$
which is of class $\cC^2$. Its derivative is
$$J'_\eps(u)v=\irn(\nabla u\cdot\nabla v + \eps^2 uv) - \irn Q(x)|u|^{p-2}uv.$$
The nontrivial critical points of $J_\eps$ belong to the Nehari manifold
$$\cN_\eps:=\{u\in H^1(\rn):u\neq 0, \ J'_\eps(u)u=0\},$$
which is a Hilbert submanifold of $H^1(\rn)$ of class $\cC^2$ and a natural constraint for $J_\eps$. Note that
\begin{equation} \label{eq:energy_on_nehari}
J_\eps(u)=\frac{p-2}{2p}\|u\|_\eps^2=\frac{p-2}{2p}\irn Q(x)|u|^p\qquad\text{if \ }u\in\cN_\eps.
\end{equation}
Since $\|\cdot\|_\eps^2$ a norm in $H^1(\rn)$, equivalent to the standard one, using Sobolev's inequality we see that 
$$c_\eps:=\inf_{u\in\cN_\eps}J_\eps(u)>0.$$
Not every element of $H^1(\rn)$ admits a radial projection onto $\cN_\eps$. Those that do belong to the set
\begin{equation} \label{eq:U}
\cU:=\{u\in L^p(\rn):\irn Q(x)|u|^p>0\}. 
\end{equation}

\begin{lemma}\label{lem:t_u}
For any $u\in H^1(\rn)\cap\cU$ there exists a unique $t_{u}\in(0,\infty)$ such that $t_{u}u\in\cN_\eps$. Explicitly,
$$t_{u}=\left(\frac{\|u\|_\eps^2}{\irn Q(x)|u|^p}\right)^\frac{1}{p-2}\qquad\text{and}\qquad J_\eps(t_uu)=\frac{p-2}{2p}\left(\frac{\|u\|_\eps^2}{\Big(\irn Q(x)|u|^p\Big)^{2/p}}\right)^\frac{p}{p-2}.$$
The function $J_{\eps,u}(t):=J_\eps(tu)$ is strictly increasing in $[0,t_{u}]$ and strictly decreasing in $[t_{u},\infty)$.
\end{lemma}

\begin{proof}
It suffices to observe that $J_{\eps,u}(t)=at^2-bt^p$ with
$$a:=\frac{1}{2}\|u\|_\eps^2>0\qquad\text{and}\qquad b:=\frac{1}{p}\irn Q(x)|u|^p>0,$$
for $t\in[0,\infty)$.
\end{proof}

Given $u\in\cN_\eps$, let $\nabla_{\cN_\eps}J_\eps(u)\in H^1(\rn)$ denote the gradient of the restriction $J_\eps|_{\cN_\eps}:\cN_\eps\to\r$ of $J_\eps$ to $\cN_\eps$ with respect to the inner product~\eqref{eq:eps_inner product}, i.e., $\nabla_{\cN_\eps}J_\eps(u)$ is the orthogonal projection of $\nabla J_\eps$ onto the tangent space to $\cN_\eps$ at $u$ where $\nabla J_\eps(u)\in H^1(\rn)$ is given by
$$\langle \nabla J_\eps(u),v\rangle_\eps=J'_\eps(u)v\qquad\text{for all \ }v\in H^1(\rn).$$
Recall that $(u_k)$ is a Palais-Smale sequence for $J_\eps|_{\cN_\eps}$ if
$$u_k\in\cN_\eps,\qquad J_\eps(u_k)\to c, \qquad \nabla_{\cN_\eps}J_\eps(u_k)\to 0,$$
and $J_\eps$ is said to satisfy the Palais-Smale condition on $\cN_\eps$ if any such sequence contains a convergent subsequence.

\begin{lemma}
Let $u_k\in\cN_\eps$ satisfy $J_\eps(u_k)\to c$ and $\nabla_{\cN_\eps}J_\eps(u_k)\to 0$. Then $(u_k)$ is bounded in $H^1(\rn)$ and $\nabla J_\eps(u_k)\to 0$.
\end{lemma}

\begin{proof}
It follows immediately from~\eqref{eq:energy_on_nehari} that $(u_k)$ is bounded.

To prove the second statement we write $\cN_\eps=F_\eps^{-1}(0)$ where $F_\eps:H^1(\rn)\cap \mathcal U \to\r$ is given by
$$F_\eps(u):=\|u\|^2_\eps - \irn Q(x)|u|^p.$$
Using Hölder’s and Sobolev’s inequalities, for every $v\in H^1(\rn)$ we obtain
\begin{equation*}
|\langle \nabla F_\eps(u_k),v\rangle_{\eps}|=\Big|\langle u_k,v\rangle_\eps - p\irn Q(x)|u_k|^{p-2}u_kv\Big| \leq (\|u_k\|_\eps + \|u_k\|_\eps^{p-1})\|v\|_\eps
\end{equation*}
As $(u_k)$ is bounded, this implies that $(\nabla F_\eps(u_k))$ is also bounded in $H^1(\rn)$. Next, we express $\nabla J_\eps(u_k)$ as the sum of its tangent and normal components, i.e.,
\begin{equation} \label{eq:sum}
\nabla J_\eps(u_k)=\nabla_{\cN_\eps}J_\eps(u_k) + t_k\nabla F_\eps(u_k),\qquad t_k\in\r.
\end{equation} 
As $u_k\in\cN_\eps$, taking the inner product with $u_k$ we get
$$0=\langle\nabla J_\eps(u_k),u_k\rangle_{\eps}=\langle\nabla_{\cN_\eps}J_\eps(u_k),u_k\rangle_{\eps} + t_k\langle\nabla F_\eps(u_k),u_k\rangle_{\eps}.$$
Since $(u_k)$ and $(\nabla F_\eps(u_k))$ are bounded in $H^1(\rn)$ and $\nabla_{\cN_\eps}J_\eps(u_k)\to 0$, it follows that $t_k\to 0$. Then, we derive from~\eqref{eq:sum} that $\nabla J_\eps(u_k)\to 0$, as claimed.
\end{proof}

\begin{lemma} \label{lem:ps}
$J_\eps$ satisfies the Palais-Smale condition on $\cN_\eps$. 
\end{lemma}

\begin{proof}
Let $u_k\in\cN_\eps$ satisfy $J_\eps(u_k)\to c$ and $\nabla_{\cN_\eps}J_\eps(u_k)\to 0$. Then $(u_k)$ is a bounded sequence in $H^1(\rn)$ and, passing to a subsequence, $u_k\rh u$ weakly in $H^1(\rn)$, $u_k\to u$ in $L_{loc}^p(\rn)$ and $u_k\to u$ a.e. in $\rn$. Since $\nabla J_\eps(u_k)\to 0$, for every $\vp\in\cC^\infty_c(\rn)$ we get that
\begin{align*}
0=\lim_{k\to\infty}\left(\langle u_k,\vp\rangle_\eps - \irn Q(x)|u_k|^{p-2}u_k\vp\right)=\langle u,\vp\rangle_\eps - \irn Q(x)|u|^{p-2}u\vp.
\end{align*}
This shows that $u$ solves~\eqref{eq:Q_problem}. Using that $\o$ is bounded and Fatou's lemma we see that
$$\frac{2p}{p-2}c_\eps\leq\lim_{k\to\infty}\irn Q(x)|u_k|^p=\lim_{k\to\infty}\io |u_k|^p-\lim_{k\to\infty}\int_{\rn\smallsetminus\o}|u_k|^p\leq \io |u|^p-\int_{\rn\smallsetminus\o}|u|^p=\irn Q(x)|u|^p.$$
This shows that $u\neq 0$. Hence, $u\in\cN_\eps$ and, from
\begin{align*}
\|u\|_\eps^2\leq \lim_{k\to\infty}\|u_k\|_\eps^2=\lim_{k\to\infty}\irn Q(x)|u_k|^p\leq\irn Q(x)|u|^p=\|u\|_\eps^2,
\end{align*}
we get that $u_k\to u$ strongly in $H^1(\rn)$, as claimed.
\end{proof}

\begin{lemma} \label{lem:regularity}
Every solution of~\eqref{eq:Q_problem} belongs to $W^{2,s}_{loc}(\rn)\cap \cC^{1,\alpha}_{loc}(\rn)\cap \cC_{loc}^\infty(\Omega)\cap \cC_{loc}^\infty(\rn\smallsetminus \overline{\Omega})$ for all $s\in[1,\infty)$ and $\alpha\in(0,1)$.
\end{lemma}

\begin{proof}
Let $u\in H^1(\rn)$ be a solution of~\eqref{eq:Q_problem}. Setting $a(x)=\eps^2+|u(x)|^{p-2}$ we see that $|-\eps^2u+Q(x)|u|^{p-2}u|\leq a(x)|u|$ and $a\in L^\frac{N}{2}_{loc}(\rn)$. Then, a Brezis-Kato-Moser argument shows that $u\in L^s_{loc}(\rn)$ for every $s\in[1,\infty)$ (see \cite[Lemma B.3]{s}), and standard elliptic regularity arguments yield that $u\in \cC^{1,\alpha}_{loc}(\rn)$ for every $\alpha\in(0,1)$ and $u\in \cC_{loc}^\infty(\Omega)\cap \cC_{loc}^\infty(\rn\smallsetminus \overline{\Omega})$ (see \cite[Appendix B]{s}).
\end{proof}

\begin{theorem} \label{thm:multiplicity}
Problem~\eqref{eq:Q_problem} has a positive least energy solution and a sequence of solutions that is unbounded in $H^1(\rn)$.
\end{theorem}

\begin{proof}
The existence of a least energy solution follows from Ekeland's variational principle and Lemma~\ref{lem:ps}. Replacing it by its absolute value, we obtain a least energy solution such that $u\geq 0$ in $\rn$. As shown in Lemma~\ref{lem:regularity}, $u\in W^{2,N}_{loc}(\rn)$. Thus, the strong maximum principle for strong solutions \cite[Theorem 9.6]{gt} applied to the equation
$$
-\Delta u + (\eps^2 - Q^-(x)|u|^{p-2})u = Q^+(x)|u|^{p-2}u,
$$
with $Q^+:=\max\{Q,0\}$ and  $Q^-:=\min\{Q,0\}$, yields that $u>0$ in $\rn$.

To prove the second statement, note that $\cN_\eps$ is symmetric with respect to the origin, $0\notin\cN_\eps$, and $J_\eps$ is even, bounded from below on $\cN_\eps$ and satisfies the Palais-Smale condition on $\cN_\eps$. Furthermore, the map
$$\Sigma_\o:=\{u\in H_0^1(\o):\|u\|_\eps=1\}\to\cN_\eps,\qquad u\mapsto \left(\frac{\|u\|_\eps^2}{\io |u|^p}\right)^\frac{1}{p-2}u,$$
is continuous and odd. Therefore, $\infty=\gen(\Sigma_\o)\leq\gen(\cN_\eps)$, where ``$\gen$" stands for the Kransnoselskii genus. It follows from \cite[Theorem II.5.7]{s} that $J_\eps$ has infinitely many pairs of critical points on $\cN_\eps$ and following the argument in \cite[Proposition 9.33]{r} one shows that $J_\eps$ has an unbounded sequence of critical values on $\cN_\eps$, as claimed.
\end{proof}
\medskip

\begin{proof}[Proof of Theorem~\ref{thm:main_multiplicity}]
Since the problems~\eqref{eq:Q_eps_problem} and~\eqref{eq:Q_problem} are equivalent, this is a consequence of Theorem~\ref{thm:multiplicity}.
\end{proof}

\section{The limit profile of positive least energy solutions}\label{poslim:sec}

Let 
\begin{align*}
E:=D^{1,2}(\rn)\cap L^p(\rn)    
\end{align*}
be the Banach space whose norm is given by
$$\|w\|_E^2:=\|u\|^2+|u|_p^2,\qquad\text{where \ }\|u\|^2:=\irn|\nabla u|^2\text{ \ and \ }|u|_p^p:=\irn|u|^p.$$
When $\o$ is the unit ball, the behavior of positive least energy solutions to~\eqref{eq:Q_problem} as $\eps\to 0$ is described in \cite[Theorem 1.1]{fw}. It is shown that, after passing to a subsequence they converge strongly in $E$ to a positive least energy solution to the problem~\eqref{eq:limit_problem}. The general case is similar. We give the details.

The solutions to~\eqref{eq:limit_problem} are the critical points of the functional $J_0:E\to\r$ given by
\begin{align}\label{J0:def}
J_0(u):=\frac{1}{2}\irn|\nabla u|^2 - \frac{1}{p}\irn Q(x)|u|^p,
\end{align}
which is of class $\cC^2$. Its derivative is
$$J'_0(u)v=\irn\nabla u\cdot\nabla v - \irn Q(x)|u|^{p-2}uv.$$
The nontrivial critical points of $J_0$ belong to the Nehari manifold
\begin{align}\label{N0:def}
\cN_0:=\{u\in E:u\neq 0, \ J'_0(u)u=0\},    
\end{align}
which is a Banach submanifold of $E$ of class $\cC^2$ and a natural constraint for $J_0$. Note that
$$J_0(u)=\frac{p-2}{2p}\irn|\nabla u|^2\qquad\text{if \ }u\in\cN_0.$$
Set 
$$c_0:=\inf_{u\in\cN_0}J_0(u).$$

\begin{lemma}
$c_0>0$.
\end{lemma}

\begin{proof}
Let $u\in\cN_0$. Since $\o$ is bounded, using Hölder's and Sobolev's inequalities we obtain
$$\|u\|^2=\irn Q(x)|u|^p\leq \io|u|^p\leq C\left(\io|u|^{2^*}\right)^{p/2^*}\leq C\left(\irn|u|^{2^*}\right)^{p/2^*}\leq C\|u\|^p. $$
Therefore,
$$C^{-1}\leq\|u\|^{p-2}\qquad\text{for every \ }u\in\cN_0,$$
and the claim follows.

\end{proof}

\begin{lemma} \label{lem:c_eps to c_0}
$\lim_{\eps\to 0}c_\eps = c_0$.
\end{lemma}

\begin{proof}
Since $H^1(\rn)\subset E$, using Lemma~\ref{lem:t_u} we obtain
$$c_\eps=\inf_{u\in H^1(\rn)\cap\,\cU}\frac{p-2}{2p}\left(\frac{\|u\|_\eps^2}{\Big(\irn Q(x)|u|^p\Big)^{2/p}}\right)^\frac{p}{p-2}\geq \inf_{u\in E\cap\,\cU}\frac{p-2}{2p}\left(\frac{\|u\|^2}{\Big(\irn Q(x)|u|^p\Big)^{2/p}}\right)^\frac{p}{p-2}=c_0,$$
with $\mathcal U$ as defined in~\eqref{eq:U}. Hence, $\liminf_{\eps\to 0}c_\eps\geq c_0$.

Next, let $(\vp_k)$ be a sequence in $\cC^\infty_c(\rn)\cap \cN_0$ such that $J_0(\vp_k)\to c_0$. Set $t_{\eps,k}:=\left(\frac{\|\vp_k\|_\eps^2}{\irn Q(x)|\vp_k|^p}\right)^{1/(p-2)}$. Then $t_{\eps,k}\vp_k\in\cN_\eps$ and, for each fixed $k$, we have that $\lim_{\eps\to 0}t_{\eps,k}=1$ and
$$\limsup_{\eps\to 0}c_\eps\leq \limsup_{\eps\to 0}\frac{p-2}{2p}\|t_{\eps,k}\vp_k\|_\eps^2=\frac{p-2}{2p}\|\vp_k\|^2=J_0(\vp_k).$$
Letting $k\to\infty$ we get that $\limsup_{\eps\to 0}c_\eps\leq c_0$. This completes the proof.
\end{proof}

\begin{lemma} \label{lem:regularity_limit}
Every solution of~\eqref{eq:limit_problem} belongs to $W^{2,s}_{loc}(\rn)\cap \cC^{1,\alpha}_{loc}(\rn)\cap \cC_{loc}^\infty(\Omega)\cap \cC_{loc}^\infty(\rn\smallsetminus \overline{\Omega})$ for all $s\in[1,\infty)$ and $\alpha\in(0,1)$.
\end{lemma}

\begin{proof}
Argue as in Lemma~\ref{lem:regularity} with $\eps=0$.
\end{proof}

The next theorem shows, in particular, that the limit problem~\eqref{eq:limit_problem} has a least energy solution.

\begin{theorem} \label{thm:positive profile}
Let $u_\eps$ be a positive least energy solution to~\eqref{eq:Q_problem}. Then, after passing to a subsequence, $(u_\eps)$ converges strongly in $E$ to a positive least energy solution of~\eqref{eq:limit_problem} as $\eps\to 0$.
\end{theorem}

\begin{proof}
Let $\eps_k\to 0$ and $u_k\in\cN_{\eps_k}$ satisfy $J_{\eps_k}(u_k)=c_{\eps_k}$ and $u_k>0$. By Lemma~\ref{lem:c_eps to c_0}, $(u_k)$ is bounded in $D^{1,2}(\rn)$ and in $L^p_{loc}(\rn)$. Since
\begin{equation}\label{eq:Lp_bounded}
\frac{2p}{p-2}c_{\eps_k}=\irn Q(x)|u_k|^p=\io |u_k|^p - \int_{\rn\smallsetminus\o}|u_k|^p,
\end{equation}
we have that
$$\int_{\rn\smallsetminus\o}|u_k|^p\leq\io |u_k|^p.$$
Hence, $(u_k)$ is bounded in $L^p(\rn)$ and, thus, in $E$. So, after passing to a subsequence, $u_k\rh u$ weakly in $E$, $u_k\to u$ in $L_{loc}^p(\rn)$ and $u_k\to u$ a.e. in $\rn$. Then, $u\geq 0$. Passing to the limit in~\eqref{eq:Lp_bounded} we see that
$$\frac{2p}{p-2}c_0\leq\io |u|^p.$$
This shows that $u\neq 0$. Using Fatou's lemma we obtain
\begin{align} \label{eq:t<1}
\|u\|^2&\leq\liminf_{k\to\infty}\|u_k\|^2\leq\lim_{k\to\infty}\|u_k\|_{\eps_k}^2=\lim_{k\to\infty}\irn Q(x)|u_k|^p \\
&=\lim_{k\to\infty}\io |u_k|^p - \lim_{k\to\infty}\int_{\rn\smallsetminus\o}|u_k|^p\leq\io |u|^p - \int_{\rn\smallsetminus\o}|u|^p=\irn Q(x)|u|^p. \nonumber
\end{align}
Therefore, there exists $t\in(0,1]$ such that $\|tu\|^2=\irn Q(x)|tu|^p$. Then, $tu\in\cN_0$ and as a consequence
$$\frac{2p}{p-2}c_0\leq\|tu\|^2\leq\|u\|^2\leq\liminf_{k\to\infty}\|u_k\|^2\leq\lim_{k\to\infty}\|u_k\|_{\eps_k}^2=\frac{2p}{p-2}c_0.$$
This shows that
\begin{equation} \label{eq:eps u^2}
\irn \eps_k^2|u_k|^2\to 0,
\end{equation}
that $u_k\to u$ strongly in $D^{1,2}(\rn)$ and that $t=1$. Hence, $\|u\|^2=\irn Q(x)|u|^p=\frac{2p}{p-2}c_0$. Combining this identity with~\eqref{eq:t<1} we see that $u_k\to u$ strongly in $L^p(\rn)$ and, thus, in $E$. As $u\in\cN_0$ and $J_0(u)=c_0$, $u$ is a least energy solution to~\eqref{eq:limit_problem}. Since $u\geq 0$, from Lemma~\ref{lem:regularity_limit} and the maximum principle (see, for instance, \cite[Theorem 9.6]{gt}) we derive that $u>0$ in $\rn$.
\end{proof}
\medskip

\begin{proof}[Proof of Theorem~\ref{thm:main_positive}]
Since $u(x):=\eps^\frac{2}{p-2}v(\eps x)$ is a solution of~\eqref{eq:Q_problem} if and only if $v$ is a solution of~\eqref{eq:Q_eps_problem}, the first statements of Theorem~\ref{thm:main_positive} follow from Theorems~\ref{thm:multiplicity} and~\ref{thm:positive profile}. Next, fix $\vr>0$. Since $(u_n)$ converges strongly in $D^{1,2}(\rn)$, performing a change of variable and using~\eqref{eq:eps u^2} we obtain
$$\frac{\int_{|x|\leq\vr}(|\nabla v_n|^2+v_n^2)}{\irn (|\nabla v_n|^2+v_n^2)} = \frac{\eps_n^{N-\frac{2p}{p-2}}\int_{|x|\leq\frac{\vr}{\eps_n}}(|\nabla u_n|^2+\eps_n^2u_n^2)}{\eps_n^{N-\frac{2p}{p-2}}\irn (|\nabla u_n|^2+\eps_n^2u_n^2)}  = \frac{\int_{|x|\leq\frac{\vr}{\eps_n}}(|\nabla u_n|^2+\eps_n^2u_n^2)}{\irn (|\nabla u_n|^2+\eps_n^2u_n^2)}\to 1.$$
And, as $(u_n)$ converges strongly in $L^p(\rn)$, we get that
$$\frac{\int_{|x|\leq\vr}|v_n|^p}{\irn |v_n|^p} = \frac{\eps_n^{N-\frac{2p}{p-2}}\int_{|x|\leq\frac{\vr}{\eps_n}}|u_n|^p}{\eps_n^{N-\frac{2p}{p-2}}\irn |u_n|^p}  = \frac{\int_{|x|\leq\frac{\vr}{\eps_n}}|u_n|^p}{\irn |u_n|^{p}} \to 1.$$
This completes the proof.
\end{proof}

\section{Existence of least energy nodal solutions}\label{n:sec}

Our next goal is to investigate the existence of least energy nodal solutions to~\eqref{eq:Q_problem} and to describe their behavior as $\eps\to 0$. To this end, we consider the set
$$\cE_\eps:=\{u\in\cN_\eps:u^+\in\cN_\eps\text{ \ and \ }u^-\in\cN_\eps\},$$
where $u^+:=\max\{u,0\}$ and $u^-:=\min\{u,0\}$. The nodal solutions to~\eqref{eq:Q_problem} belong to $\cE_\eps$. Let
$$F_\eps(u):=J'_\eps(u)u=\|u\|_\eps^2-\irn Q(x)|u|^p, \qquad u\in H^1(\rn),$$
$\cU:=\{v\in L^p(\rn):\irn Q(x)|v|^p>0\}$ and $\what{\cN}_\eps:=\{w\in\cN_\eps:w^+\in\cU, \ w^-\in\cU\}$. Note that $\cE_\eps\subset \what{\cN}_\eps$. Define
$$\what{\cN}_\eps^+:=\{w\in\what{\cN}_\eps:F_\eps(w^+)<0\}\qquad\text{and}\qquad\what{\cN}_\eps^-:=\{w\in\what{\cN}_\eps:F_\eps(w^+)>0\}.$$
The following lemma is obtained by adapting the argument of \cite[Lemma 2.4]{ccn}.

\begin{lemma} \label{lem:gamma}
Given $u\in\cE_\eps$ there exists a continuous map $\gamma:[0,1]\to\cN_\eps$ such that $\gamma(0)=u^+$, $\gamma(1)=u^-$, $\gamma(\frac{1}{2})=u$,
$$0<J_\eps(\gamma(s))<J_\eps(u)\qquad\text{if \ }s\in[0,\tfrac{1}{2})\cup(\tfrac{1}{2},1]$$
and
$$\gamma(s)\in\what{\cN}_\eps^+\text{ \ iff \ }s\in(0,\tfrac{1}{2}) \qquad\text{and}\qquad \gamma(s)\in\what{\cN}_\eps^-\text{ \ iff \ }s\in(\tfrac{1}{2},1).$$
\end{lemma}

\begin{proof}
Given $u\in\cE_\eps$, define $\gamma=\gamma_u:[0,1]\to\cN_\eps$ as
$$\gamma(s):= t_{u_s}u_s,\qquad\text{where}\quad u_s:=(1-s)u^++su^-$$
and $t_{u_s}$ is the unique positive number such that $t_{u_s}u_s\in\cN_\eps$, as in Lemma~\ref{lem:t_u}. Note that $u_s\in\cU$ for every $s\in[0,1]$, so $t_{u_s}$ is well defined. As $u^+,u^-,u\in\cN_\eps$, this path satisfies
$$\gamma(0)=u^+,\qquad \gamma(1)=u^-,\qquad\gamma\left(\tfrac{1}{2}\right)=u.$$
Fix $s\in(0,1)$ and set $w:=\gamma(s)$. Then, $w^+=t_{u_s}(1-s)u^+$ and $w^-=t_{u_s}su^-$. Since $w\in\cN_\eps$ we have that
$$\|w^+\|^2+\|w^-\|^2=\|w\|^2=\irn Q(x)|w|^p=\irn Q(x)|w^+|^p+\irn Q(x)|w^-|^p.$$
Therefore, recalling that $u^+,u^-\in\cN_\eps$, we obtain
\begin{align*}
w\in\what{\cN}_\eps^+&\quad\Longleftrightarrow\quad \|w^+\|^2<\irn Q(x)|w^+|^p\quad\Longleftrightarrow\quad \|w^-\|^2>\irn Q(x)|w^-|^p \\
& \quad\Longleftrightarrow\quad (t_{u_s}(1-s))^{p-2}>1>(t_{u_s}s)^{p-2}\quad\Longleftrightarrow\quad s\in(0,\tfrac{1}{2}).
\end{align*}
Similarly,
$$w\in\what{\cN}_\eps^- \quad\Longleftrightarrow\quad s\in(\tfrac{1}{2},1).$$
Next, fix $s\in(0,\frac{1}{2})$. Note that $t_{w^+}=(t_{u_s}(1-s))^{-1}$.  Then, from Lemma~\ref{lem:t_u}, we deduce that $J_\eps(t_{w^+}w^-)=J_\eps(\frac{s}{1-s}u^-)>0$ and
\begin{align*}
0&<J_\eps(t_{w^+}w^+)<J_\eps(t_{w^+}w^+)+J_\eps(t_{w^+}w^-)=J_\eps(t_{w^+}w) \\
&<J_\eps(w)=J_\eps(t_{u_s}(1-s)u^+)+J_\eps(t_{u_s}su^-)<J_\eps(u^+)+J_\eps(u^-)=J_\eps(u),
\end{align*}
as claimed. The argument for $s\in(\frac{1}{2},1)$ is similar. This completes the proof.
\end{proof}

\begin{lemma} \label{lem:intersection}
If $\sigma:[a,b]\to\what{\cN}_\eps$ is continuous, $\sigma(a)\in \what{\cN}_\eps^+$ and $\sigma(b)\in \what{\cN}_\eps^-$, then there exists $t\in(a,b)$ such that $\sigma(t)\in\cE_\eps$.
\end{lemma}

\begin{proof}
Since the function $f(t):=F_\eps(\sigma(t)^+)$ is continuous in $[a,b]$, $f(a)<0$ and $f(b)>0$, there exists $t_0\in(a,b)$ such that $f(t_0)=0$. This implies that $\sigma(t_0)^+\in\cN_\eps$, because $u^+\neq 0$ for every $u\in\what{\cN}_\eps$. Observe that, if $u,u^+\in\cN_\eps$ and $u^-\neq 0$, then $u^-\in\cN_\eps$, i.e.,  $u\in\cE_\eps$. This remark shows that $\sigma(t_0)\in\cE_\eps$ and completes the proof.
\end{proof}

Set
$$\what c_\eps:=\inf_{v\in\cE_\eps}J_\eps(v).$$

\begin{lemma} \label{lem:nodal minimum}
If $u\in\cE_\eps$ and $J_\eps(u)=\what c_\eps$, then $u$ is a critical point of $J_\eps$.
\end{lemma}

\begin{proof}
Let $u\in\cE_\eps$ be such that $J_\eps(u)=\what c_\eps$ and let $\gamma:[0,1]\to\cN_\eps$ be as in Lemma~\ref{lem:gamma}.

Arguing by contradiction, assume that $J_\eps'(u)\neq 0$. Fix $\delta>0$ such that $\overline{B}_\delta(u):=\{v\in\cN_\eps:\|v-u\|_\eps\leq\delta\}\subset\what\cN_\eps$ and $\|\nabla_{\cN_\eps} J_\eps(v)\|\geq \alpha>0$ for every $v\in\overline{B}_\delta(u)$, where $\nabla_{\cN_\eps} J_\eps(v)$ is the gradient of $J_\eps(v)|_{\cN_\eps}$ at $v$. Choose $\beta\in(0,\frac{\delta\alpha}{16})$. The deformation lemma \cite[Lemma 5.15]{w} yields a map $\eta:[0,1]\times\cN_\eps\to\cN_\eps$ such that $\eta(t,u)=u$ if $t=0$ or $u\notin J_\eps^{-1}[\what c_\eps-2\beta,\what c_\eps+2\beta]\cap \overline{B}_\delta(u)$, $J_\eps(\eta(1,u))\leq \what c_\eps -\beta$ and $t\mapsto J_\eps(\eta(t,v))$ is nonincreasing for every $v\in\cN_\eps$. Furthermore, $\eta(t,v)\in \overline{B}_\delta(u)$ if $v\in \overline{B}_\delta(u)$ for every $t\in[0,1]$. Define $\sigma:[0,1]\to\cN_\eps$ by $\sigma(s):=\eta(1,\gamma(s))$. Then,
\begin{align*}
J_\eps(\sigma(s))
=J_\eps(\eta(1,\gamma(s)))
\leq
J_\eps(\eta(0,\gamma(s)))
=J_\eps(\gamma(s))<J_\eps(u)\qquad \text{ if $s\in[0,\tfrac{1}{2})\cup(\tfrac{1}{2},1]$ }
\end{align*}
 and $J_\eps(\sigma(\frac{1}{2}))=J_\eps(\eta(1,u))<J_\eps(u)$. This proves that $\sigma(t)\notin\cE_\eps$ for every $t\in[0,1]$.

Note that $\sigma(s)\in\what\cN_\eps$ if $s\in(0,1),$ because $\eta(1,\gamma(s))\in\overline{B}_\delta(u)\subset\what\cN_\eps$ if $\gamma(s)\in \overline{B}_\delta(u)$ and $\eta(1,\gamma(s))=\gamma(s)\in\what\cN_\eps$  if $\gamma(s)\notin \overline{B}_\delta(u)$ and $s\in(0,1)$. Since $\gamma^{-1}(\overline{B}_\delta(u))$ is closed in $[0,1]$ and does not contain $0$ or $1$, there exist $a\in (0,\frac{1}{4})$ and $b\in (\frac{3}{4},1)$ such that $\gamma(a),\gamma(b)\notin \overline{B}_\delta(u)$. Then, $\sigma(a)=\gamma(a)\in\what\cN_\eps^+$ and $\sigma(b)=\gamma(b)\in\what\cN_\eps^-$ and $\sigma:[a,b]\to\what{\cN}_\eps$ satisfies the hypotheses of Lemma~\ref{lem:intersection}. As a consequence, there exists $t\in(a,b)$ such that $\sigma(t)\in\cE_\eps$. This is a contradiction.
\end{proof}

A function $u\in\cE_\eps$ such that $J_\eps(u)=\what c_\eps$ is called a \emph{least energy nodal solution to}~\eqref{eq:Q_problem}. Adapting the argument in \cite[Section 3]{ccn} we prove the following result.

\begin{theorem} \label{thm:nodal existence}
Problem~\eqref{eq:Q_problem} has a least energy nodal solution.
\end{theorem}

\begin{proof}
Let $(u_k)$ be a sequence in $\cE_\eps$ such that $J_\eps(u_k)\to\what c_\eps$. Since $(u_k)$, $(u_k^+)$ and $(u_k^-)$ are bounded in $H^1(\rn)$, after passing to a subsequence,
\begin{align*}
&u_k\rh u\text{ \ weakly in \ }H^1(\rn),\qquad u_k\to u\text{ \ in \ }L_{loc}^p(\rn), \\
&u_k^+\rh v\text{ \ weakly in \ }H^1(\rn),\qquad u_k^+\to v\text{ \ in \ }L_{loc}^p(\rn), \\
&u_k^-\rh w\text{ \ weakly in \ }H^1(\rn),\qquad u_k^-\to w\text{ \ in \ }L_{loc}^p(\rn).
\end{align*}
Since the maps $z\mapsto z^+$ and $z\mapsto z^-$ are continuous in $L^p(\rn)$ we have that $v=u^+$ and $w=u^-$. And, as $u_k^+\in\cN_\eps$, using Lemma~\ref{lem:c_eps to c_0} we obtain
$$\frac{2p}{p-2}c_0\leq\liminf_{k\to\infty}\|u_k^+\|_\eps^2=\liminf_{k\to\infty}\irn Q(x)|u_k^+|^p\leq\lim_{k\to\infty}\io|u_k^+|^p=\io|u^+|^p.$$ 
Therefore, $u^+\neq 0$. Furthermore, using Fatou's lemma,
\begin{align*}
\|u^+\|_\eps^2&\leq\liminf_{k\to\infty}\|u_k^+\|_{\eps}^2=\liminf_{k\to\infty}\irn Q(x)|u_k^+|^p=\lim_{k\to\infty}\io |u_k^+|^p - \liminf_{k\to\infty}\int_{\rn\smallsetminus\o}|u_k^+|^p \\
&\leq\io |u^+|^p - \int_{\rn\smallsetminus\o}|u^+|^p=\irn Q(x)|u^+|^p. \nonumber
\end{align*}
Hence, there exists $t_{u^+}\in(0,1]$ such that $\|t_{u^+}u^+\|_\eps^2=\irn Q(x)|t_{u^+}u^+|^p$. Similarly, $u^-\neq 0$ and there exists $t_{u^-}\in(0,1]$ such that $\|t_{u^-}u^-\|_\eps^2=\irn Q(x)|t_{u^-}u^-|^p$. It follows that $w:=t_{u^+}u^+ + t_{u^-}u^-\in\cE_\eps$. Thus,
\begin{align*}
\frac{2p}{p-2}\what c_\eps\leq\|w\|_\eps^2=\|t_{u^+}u^+\|_\eps^2+\|t_{u^-}u^-\|_\eps^2 \leq \|u^+\|_\eps^2+\|u^-\|_\eps^2 = \|u\|_\eps^2\leq\lim_{k\to\infty}\|u_k\|_{\eps}^2=\frac{2p}{p-2}\what c_\eps.
\end{align*}
As a consequence, $t_{u^+}=1=t_{u^-}$, $u\in\cE_\eps$ and $J_\eps(u)=\what c_\eps$, i.e., $u$ is a least energy nodal solution of~\eqref{eq:Q_problem}.
\end{proof}

\section{The limit profile of least energy nodal solutions}\label{nodlimit:sec}

Recall the definition of $\cN_0$ given in~\eqref{N0:def}.  The nodal solutions to~\eqref{eq:limit_problem} belong to the set
\begin{align}\label{E0:def}
\cE_0:=\{u\in\cN_0:u^+\in\cN_0\text{ \ and \ }u^-\in\cN_0\}.    
\end{align}
Let
$$\what c_0:=\inf_{v\in\cE_0}J_0(v).$$

\begin{lemma}
If $u\in\cE_0$ and $J_0(u)=\what c_0$, then $u$ is a critical point of $J_0$.
\end{lemma}

\begin{proof}
The proof is the same as that of Lemma~\ref{lem:nodal minimum}.
\end{proof}

A function $u\in\cE_0$ such that $J_0(u)=\what c_0$ is called a \emph{least energy nodal solution to~\eqref{eq:limit_problem}}.

\begin{lemma} \label{lem:wc_eps to wc_0}
$\lim_{\eps\to 0}\what c_\eps=\what c_0$.
\end{lemma}

\begin{proof}
Let $\what H:=\{u\in H^1(\rn):u^+,u^-\in \cU\}$ with $\cU$ as in~\eqref{eq:U}. Note that, if $u\in\what H$ and $t_{\eps, u^+},\,t_{\eps, u^-}\in(0,\infty)$ are such that $t_{\eps,u^+}u^+,t_{\eps,u^-}u^-\in\cN_\eps$, then $\what u_\eps:=t_{\eps,u^+}u^+ + t_{\eps,u^-}u^-\in\cE_\eps$ and $J_\eps(\what u_\eps)=J_\eps(t_{\eps,u^+}u^+) + J_\eps(t_{\eps,u^-}u^-)$. Using Lemma~\ref{lem:t_u} we obtain
\begin{align*}
\what c_\eps &=\inf_{u\in \what H}J_\eps(\what u_\eps)=\inf_{u\in \what H}(J_\eps(t_{\eps,u^+}u^+) + J_\eps(t_{\eps,u^-}u^-)) \\
&=\inf_{u\in \what H}\left(\frac{p-2}{2p}\left(\frac{\|u^+\|_\eps^2}{\Big(\irn Q(x)|u^+|^p\Big)^{2/p}}\right)^\frac{p}{p-2} + \frac{p-2}{2p}\left(\frac{\|u^-\|_\eps^2}{\Big(\irn Q(x)|u^-|^p\Big)^{2/p}}\right)^\frac{p}{p-2}\right).
\end{align*}
Similarly, setting $\what E:=\{u\in E:u^+,u^-\in\cU\}$ we get
$$\what c_0=\inf_{u\in \what E}\left(\frac{p-2}{2p}\left(\frac{\|u^+\|^2}{\Big(\irn Q(x)|u^+|^p\Big)^{2/p}}\right)^\frac{p}{p-2} + \frac{p-2}{2p}\left(\frac{\|u^-\|^2}{\Big(\irn Q(x)|u^-|^p\Big)^{2/p}}\right)^\frac{p}{p-2}\right).$$
Since $\|u\|_\eps\geq\|u\|$ for every $u\in H^1(\rn)$ and $\what H\subset\what E$, we have that $\what c_\eps\geq\what c_0$ for every $\eps>0$. Therefore, $\liminf_{\eps\to 0}\what c_\eps\geq\what c_0$. 

Let $(\vp_k)$ be a sequence in $\cC^\infty_c(\rn)\cap \cE_0$ such that $J_0(\vp_k)\to\what c_0$. Set $t_{\eps,k}^\pm:=\left(\frac{\|\vp_k^\pm\|_\eps^2}{\irn Q(x)|\vp_k^\pm|^p}\right)^{1/(p-2)}$. Then $t_{\eps,k}^\pm\vp_k^\pm\in\cN_\eps$ and, for each fixed $k$, we have that $\lim_{\eps\to 0}t_{\eps,k}^\pm=1$ and
$$\limsup_{\eps\to 0}\what c_\eps\leq \limsup_{\eps\to 0}\frac{p-2}{2p}\left(\|t_{\eps,k}^+\vp_k^+\|_\eps^2 + \|t_{\eps,k}^-\vp_k^-\|_\eps^2\right)=\frac{p-2}{2p}\left(\|\vp_k^+\|^2 + \|\vp_k^-\|^2\right)=J_0(\vp_k).$$
Letting $k\to\infty$ we get that $\limsup_{\eps\to 0}\what c_\eps\leq\what c_0$. This completes the proof.
\end{proof}

The following result yields the existence of a least energy nodal solution to the limit problem~\eqref{eq:limit_problem}.

\begin{theorem} \label{thm:nodal profile}
Let $u_\eps$ be a least energy nodal solution to~\eqref{eq:Q_problem}. Then, after passing to a subsequence, $(u_\eps)$ converges strongly in $E$ to a least energy nodal solution of~\eqref{eq:limit_problem} as $\eps\to 0$.
\end{theorem}

\begin{proof}
Let $\eps_k\to 0$ and $u_k\in\cE_{\eps_k}$ satisfy $J_{\eps_k}(u_k)=\what c_{\eps_k}$. Arguing as in the proof of Theorem~\ref{thm:positive profile}, using Lemma~\ref{lem:wc_eps to wc_0}, we see that $(u_k)$ is bounded in $E$. So, after passing to a subsequence, $u_k\rh u$ weakly in $E$ and $u_k\to u$ in $L_{loc}^p(\rn)$. Furthermore, as in the proof of Theorem~\ref{thm:nodal existence}, we see that $u_k^\pm\rh u^\pm$ weakly in $E$ and $u_k^\pm\to u^\pm$ in $L_{loc}^p(\rn)$. Using Lemma~\ref{lem:c_eps to c_0} and Fatou's lemma we obtain
\begin{align*} 
\frac{2p}{p-2}c_0&\leq\lim_{k\to\infty}\|u_k^\pm\|_{\eps_k}^2=\lim_{k\to\infty}\irn Q(x)|u_k^\pm|^p \\
&=\lim_{k\to\infty}\io |u_k^\pm|^p - \lim_{k\to\infty}\int_{\rn\smallsetminus\o}|u_k^\pm|^p\leq\io |u^\pm|^p - \int_{\rn\smallsetminus\o}|u^\pm|^p=\irn Q(x)|u^\pm|^p.
\end{align*}
This shows that $u^\pm\in\cU$. As 
$$\|u^\pm\|^2\leq\liminf_{k\to\infty}\|u_k^\pm\|^2\leq\lim_{k\to\infty}\|u_k^\pm\|_{\eps_k}^2\leq\irn Q(x)|u^\pm|^p,$$
there exist $t_\pm\in(0,1]$ such that $\|t_\pm u^\pm\|^2=\irn Q(x)|t_\pm u^\pm|^p$. Then, $w:=t_+u^+ + t_-u^-\in\cE_0$ and, as a consequence,
\begin{align*}
\frac{2p}{p-2}\what c_0 &\leq\|w\|^2=\|t_+u^+\|^2 + \|t_-u^-\|^2\leq\|u^+\|^2 + \|u^-\|^2=\|u\|^2 \\
&\leq\liminf_{k\to\infty}\|u_k\|^2\leq\lim_{k\to\infty}\|u_k\|_{\eps_k}^2=\frac{2p}{p-2}\what c_0.
\end{align*}
This shows that $u_k\to u$ strongly in $D^{1,2}(\rn)$ and that $t_+=1=t_-$. Hence, $u\in\cE_0$ and $\|u\|^2=\irn Q(x)|u|^p=\frac{2p}{p-2}\what c_0$. By Lemma~\ref{lem:wc_eps to wc_0} and Fatou's lemma,
\begin{align*} 
\frac{2p}{p-2}\what c_0 &=\lim_{k\to\infty}\irn Q(x)|u_k|^p =\lim_{k\to\infty}\io |u_k|^p - \lim_{k\to\infty}\int_{\rn\smallsetminus\o}|u_k|^p \\
&\leq\io |u|^p - \int_{\rn\smallsetminus\o}|u|^p=\irn Q(x)|u|^p=\frac{2p}{p-2}\what c_0.
\end{align*}
This shows that $u_k\to u$ strongly in $L^p(\rn)$ and, thus, in $E$. This completes the proof.
\end{proof}
\medskip

\begin{proof}[Proof of Theorem~\ref{thm:main_nodal}]
Since $u(x):=\eps^\frac{2}{p-2}v(\eps x)$ is a solution of~\eqref{eq:Q_problem} if and only if $v$ is a solution of~\eqref{eq:Q_eps_problem}, the first statements of Theorem~\ref{thm:main_nodal} follow from Theorems~\ref{thm:nodal existence} and~\ref{thm:nodal profile}. The proof of~\eqref{eq:concentration} goes as in the proof of Theorem~\ref{thm:main_positive}.
\end{proof}

\section{The symmetries of the limit profiles}\label{s:sec}

\subsection{Radial symmetry of the positive least energy solution}

We recall the notion of Schwarz symmetrization, also called symmetric rearrangement (see \cite[Chapter 1]{k} or \cite{ll}). Let $A$ be a measurable set of finite volume $|A|$ in $\rn$. Its \emph{Schwarz symmetrization} $A^*$ is the open ball in $\rn$ centered at the origin whose volume is the same as the volume of $A$. Let $u:\rn\to\r$ be a nonnegative measurable function that vanishes at infinity, in the sense that its level sets have finite volume, i.e.,
$$\{u>t\}:=\{x\in\rn:u(x)>t\}$$
has finite volume for every $t>0$. The \emph{Schwarz symmetrization} $u^*:\rn\to\r$ of $u$ is defined by symmetrizing its level sets, i.e.,
$$u^*(x):=\int_0^\infty\chi_{\{u>t\}^*}(x)\d t.$$
Here $\chi_A$ stands, as usual, for the characteristic function of $A$. Clearly, $u^*$ is nonnegative, radially symmetric and nonincreasing in the radial direction. Since $\{u>t\}$ is open, $u^*$ is lower semicontinuous and, therefore, measurable. By construction, $u^*$ and $u$ are \emph{equimeasurable}, i.e.,
$$|\{u^*>t\}|=|\{u>t\}|\qquad\text{for all \ }t>0.$$
Then, using Fubini's theorem and the layer-cake decomposition of $u$, given by
$$u(x)=\int_0^\infty\chi_{\{u>t\}}(x)\d t,$$
one easily concludes that $u^*\in L^p(\rn)$ if $u\in L^p(\rn)$ and that their norms in $L^p(\rn)$ are the same, i.e.,
$$|u|_p=|u^*|_p.$$
These and other properties may be found in \cite[Section 3.3]{ll}.

Recall the definitions of $J_0$ and $\cN_0$ given in~\eqref{J0:def} and~\eqref{N0:def}.

\begin{lemma}\label{lem:radially}
Let $w$ be a positive least energy solution of~\eqref{eq:limit_problem} with $\Omega = B_1$. Then $w^*\in\cN_0$ is a least energy solution of~\eqref{eq:limit_problem}, $w^*$ is strictly decreasing in the radial direction, and
$$\int_{B_1}|w|^{p}=\int_{B_1}|w^*|^{p}.$$     
\end{lemma}

\begin{proof}Let $w$ be a positive least energy solution of~\eqref{eq:limit_problem} with $\Omega = B_1$. By \cite[Lemma 7.17]{ll} we have that $\|w\|^2\geq \|w^*\|^2$ and, as pointed out above, $|w|_p=|w^*|_p$. We claim that
\begin{align} \label{eq:claim}
\int_{\rn} Q(x) |w|^{p}\leq \int_{\rn} Q(x) |w^*|^{p}.
\end{align}
Let 
\begin{align*}
m_1:=\int_{B_1}|w|^{p},\quad 
m_2:=-\int_{\rn \smallsetminus B_1} |w|^{p},\quad 
m_1^*:=\int_{B_1}|w^*|^{p},\quad 
m_2^*:=-\int_{\rn\smallsetminus B_1} |w^*|^{p}.
\end{align*}
By definition, it is clear that $m_1^*\geq m_1$. As 
\begin{equation}\label{eq:m}
m_1-m_2=|w|_p^p=|w^*|_p^p=m_1^*-m_2^*
\end{equation}
we have that $m_2^*\geq m_2$. This proves~\eqref{eq:claim}. As $w\in\cN_0$ we obtain
$$\int_{\rn} Q(x) |w^*|^{p}\geq\int_{\rn} Q(x) |w|^{p}=\|w\|^2\geq \|w^*\|^2.$$
Arguing as in Lemma~\ref{lem:t_u} we see that there exists $t_*\in(0,1]$ such that $t_*w^*\in\cN_0$. Therefore,
$$c_0\leq J_0(t_*w^*)=\frac{p-2}{2p}\|t_*w^*\|^2\leq\frac{p-2}{2p}\|w^*\|^2\leq\frac{p-2}{2p}\|w\|^2=J_0(w)=c_0.$$
This implies that $t_*=1$ and that $w^*\in\cN_0$ is a least energy solution of~\eqref{eq:limit_problem}. As a consequence,~\eqref{eq:claim} is an equality, i.e., $m_1+m_2=m_1^*+m_2^*$. This, together with~\eqref{eq:m} yields $m_1=m_1^*$. Furthermore, by Lemma~\ref{lem:regularity_limit}, $w^*$ is a strong solution of~\eqref{eq:limit_problem}. This implies that $w^*$ is strictly decreasing in the radial direction. Otherwise, it would be constant in some annulus, contradicting~\eqref{eq:limit_problem}.
\end{proof}
\medskip

\begin{proof}[Proof of Theorem~\ref{thm:symmetries}$(i)$]
Let $w$ be a positive least energy solution of~\eqref{eq:limit_problem} in $B_1$ and $w^*$ be its Schwarz symmetrization. By Lemma~\ref{lem:radially}, it suffices to show that $w=w^*$. As $\|w\|^2=\|w^*\|^2$, we have that $w$ coincides with a translate of $w^*$ \cite{bz}, i.e., $w(x)=w^*(x+\xi)$ for some $\xi\in\rn$. Since $w^*$ is radially symmetric and strictly decreasing in the radial direction we have that
\begin{align*}
\int_{B_1} |w^*|^{p} > \int_{B_1} |w^*(x+\xi)|^{p} \d x\qquad \text{ for every \ }\xi\neq 0.
\end{align*}
Thus, by Lemma~\ref{lem:radially}, $w=w^*$.
\end{proof}

\subsection{Foliated Schwarz symmetry of the least energy nodal solution}\label{fss:sec}

Let $\mathbb{S}^{N-1}=\left\{x \in \mathbb{R}^N:|x|=1\right\}$ be the unit sphere. For each $e \in \mathbb{S}^{N-1}$ we consider the halfspace $H(e):=\{x \in$ $\left.\mathbb{R}^N: x \cdot e>0\right\}$, and we denote by $u_e$ the composition of the function $u: \rn \rightarrow \mathbb{R}$ with the reflection $\sigma_e$ with respect to the hyperplane $\partial H(e)$, that is,
$$u_e: \rn\rightarrow \mathbb{R} \quad \text { is given by } \quad u_e(x):=u(\sigma_e(x)),\qquad\text{where \ } \sigma_e(x)=x-2(x \cdot e) e.$$
The \emph{polarization} $u_H$ of $u: \rn \rightarrow \mathbb{R}$ with respect to the halfspace $H=H(e)$ is defined by
$$u_H:= \begin{cases}\max \left\{u, u_e\right\} & \text { in } \overline{H}, \\ \min \left\{u, u_e\right\} & \text { in } \rn \smallsetminus H.\end{cases}$$

Recall the definition of $\cE_0$ given in~\eqref{E0:def}.

\begin{lemma} \label{lem:u_H}
Let $\o$ be radially symmetric. Then, for any halfspace $H=H(e)$, we have that $u_H\in\cN_0$ if $u\in\cN_0$, $u_H\in\cE_0$ if $u\in\cE_0$, and $J_0(u)=J_0(u_H)$. In particular, if $u$ is a least energy (positive or nodal) solution of the limit problem~\eqref{eq:limit_problem}, then so is $u_H$.
\end{lemma}

\begin{proof}
Observe that, for any $u:\rn\to\r$, we have that $(u^\pm)_H=(u_H)^\pm$ and that $(u\chi_A)_H=u_H\chi_A$ for any radially symmetric set $A$. From \cite[Corollary 3.2]{we} we have that
$$\|u\|^2=\|u_H\|^2\qquad\text{and}\qquad\|u^\pm\|^2=\|u_H^\pm\|^2$$
and, from \cite[Lemma 3.1]{we}, we derive
$$\io |u^\pm|^p = \io |u_H^\pm|^p\qquad\text{and}\qquad\int_{\rn\smallsetminus\o} |u^\pm|^p = \int_{\rn\smallsetminus\o} |u_H^\pm|^p$$
and the similar statements with $u^\pm$ replaced by $u$. Therefore,
$$\irn Q(x)|u|^p=\irn Q(x)|u_H|^p\qquad\text{and}\qquad\irn Q(x)|u^\pm|^p=\irn Q(x)|u_H^\pm|^p.$$
As a consequence, $u_H\in\cN_0$ if $u\in\cN_0$, $u_H\in\cE_0$ if $u\in\cE_0$, and $J_0(u)=J_0(u_H)$, as claimed.
\end{proof}

A function $u \in \cC^0(D)$ is said to be \emph{foliated Schwarz symmetric with respect to} $e^* \in \mathbb{S}^{N-1}$ if $u$ is axially symmetric with respect to the axis $\mathbb{R} e^*:=\{te^*:t\in \r\}$ (i.e., it is invariant under rotations around this axis) and nonincreasing in the polar angle $\theta:=\arccos \left(\frac{x}{|x|} \cdot e^*\right) \in[0, \pi]$.

We use the following characterization of foliated Schwarz symmetry.

\begin{lemma} \label{crit}
Let $\o$ be a radially symmetric open set in $\rn$. There exists $e^* \in \mathbb{S}^{N-1}$ such that $u \in \cC^0(\overline{\o})$ is foliated Schwarz symmetric with respect to $e^*$ if and only if for every $e \in \mathbb{S}^{N-1}$ either
\begin{align*}
u \geqslant u_e \quad \text {in } \o\cap H(e) \qquad \text { or } \qquad u \leqslant u_e \quad \text {in } \o\cap H(e). 
\end{align*}
\end{lemma}

\begin{proof}
The proof is given in \cite{b} (see also \cite[Proposition 3.2]{sw} and \cite{we}).  This proof only takes into account the case when $\Omega$ is connected, but the general case easily follows since every component satisfies the same inequalities which completely determine the symmetry axis and the angular monotonicity.
\end{proof}
\medskip

\begin{proof}[Proof of Theorem~\ref{thm:symmetries}$(ii)$]
First, we adapt the strategy from \cite[Lemma 2.5]{bww} to our setting. 
Let $\o$ be radially symmetric and let $u$ be a least energy (positive or nodal) solution of the limit problem~\eqref{eq:limit_problem}. Let  $e\in \mathbb{S}^{N-1}$ and set $H:=H(e)$. By Lemma~\ref{lem:u_H}, $u_H$ is a least energy (positive or nodal) solution of~\eqref{eq:limit_problem}. Observe that $\left|u-u_e\right|=2 u_H-\left(u+u_e\right)$ in $H$ and $-\left|u-u_e\right|=2 u_H-\left(u+u_e\right)$ in $\rn\smallsetminus H$. By Lemma~\ref{lem:regularity_limit}, $u,u_H\in W^{2,s}_{loc}(\rn)\cap \cC^{1,\alpha}_{loc}(\rn)$ for all $s\in[1,\infty)$ and $\alpha\in(0,1)$. Then $\left|u-u_e\right| \in W^{2,N}_{loc}(\rn)$ and
$$
\begin{aligned}
-\Delta|u- u_e|
&=2 Q(x)|u_H|^{p-2}u_H-\left[Q(x)|u|^{p-2}u+Q(x)|u_e|^{p-2}u_e\right] \\
& =Q(x)\left([|u_H|^{p-2}u_H-|u|^{p-2}u]+[|u_H|^{p-2}u_H-|u_e|^{p-2}u_e]\right)\qquad \text{ in \ }H:=H(e).
\end{aligned}
$$
Hence,
\begin{align}\label{uue}
-\Delta|u- u_e|\leq 0 \text{ \ in \ }H(e)\cap(\rn\smallsetminus\Omega)\quad \text{and}\quad
-\Delta|u- u_e|\geq 0 \text{ \ in \ }H(e)\cap \Omega\quad\text{for every \ }e\in\s^{N-1}.
\end{align}
Clearly, $|u- u_e|=0$ on $\partial H(e)$. By the maximum principle for strong solutions \cite[Theorem 9.6]{gt}, either $u \equiv u_e$ in $H(e)\cap \Omega$ or $\left|u-u_e\right|>0$ in $H(e)\cap \Omega$. Since $e$ is arbitrary, Lemma~\ref{crit} states that there exists $e^* \in \mathbb{S}^{N-1}$ such that $u$ is foliated Schwarz symmetric in $\o$ with respect to $e^*$. In particular, it is axially symmetric in $\o$ with respect to $\r e^*$. Therefore,
\begin{align*}
|u-u_{e}|=0 \text{ \ in \ } \big(H(e)\cap \Omega\big)\cup \partial H(e)\quad\text{for every \ }e\in\s^{N-1}\text{ \ with \ }e\cdot e^*=0.
\end{align*}
From~\eqref{uue} and the maximum principle, we obtain that $|u-u_{e}|\leq 0$ in $\rn\smallsetminus\Omega$, that is, $u\equiv u_{e}$ in $\rn$ for every $e\in\s^{N-1}$ such that $e\cdot e^*=0$. It follows that $u$ is axially symmetric with respect to $\mathbb{R} e^*$ in the whole of $\rn$.  As a consequence, $u$ only depends on two variables. Abusing notation, we write
\begin{align*}
 u(x)=u(r,\theta)\qquad \text{ for }x\in \rn,\quad r=|x|\in[0,\infty),\quad \text{ and }\quad
 \theta =\arccos\left(\frac{x}{|x|} \cdot e^*\right) \in[0, \pi].
\end{align*}
As $\o$ is radially symmetric there is an open bounded subset $\cO$ of $[0,\infty)$ such that $x\in\o$ if and only if $r=|x|\in\cO$. Let
\begin{align*}
 v(r,\theta):= \partial_\theta u(r,\theta)\qquad \text{ for }r>0\quad \text{ and }\quad \theta\in[0,\pi].
\end{align*}
Since $u$ is foliated Schwarz symmetric in $\o$, the function $\theta\mapsto u(r,\theta)$ is nonincreasing for each $r\in\overline{\cO}$; in particular
\begin{align*}
v(r,\theta)\leq 0\qquad \text{for all \ }r\in\partial\cO\text{ \ and \ }\theta\in[0,\pi].
\end{align*}
Furthermore, since $u$ is axially symmetric in $\rn$ with respect to $\r e^*$, 
\begin{align*}
v(r,0)=v(r,\pi)=0\qquad \text{for \ }r>0
\end{align*}
and, as $u$ is $\cC^1$ at the origin (in fact, smooth, by Lemma \ref{lem:regularity_limit}), we have, by continuity, that
\begin{align*}
v(0,\theta):=\lim_{r\to 0}v(r,\theta)=\lim_{r\to 0}v(r,0)=0\qquad \text{for \ }\theta\in (0,\pi).
\end{align*}
(This is only relevant if $0\not\in \Omega$).

Writing the Laplacian in hyperspherical coordinates (see, for instance, \cite[eq. (2.10)]{cs}) and using that $u$ is axially symmetric, we have that
\begin{align*}
 -\Delta u(x)
 =Lu(r,\theta):=
 -\frac{1}{r^{N-1}}\partial_r\left(
 r^{N-1}\partial_{r}u(r,\theta)
 \right)
 -\frac{1}{r^2 \sin^{N-2}\theta}\partial_{\theta}\left( \sin^{N-2} \theta\ \partial_{\theta}u(r,\theta)\right).
\end{align*}
Differentiating the above equality with respect to $\theta$ and interchanging derivatives, we get
\begin{align*}
 L v(r,\theta)=
 -\frac{1}{r^{N-1}}\partial_r\left(
 r^{N-1}\partial_{r}v(r,\theta)
 \right)
 -\partial_\theta\left(\frac{1}{r^2 \sin^{N-2}\theta}\partial_{\theta}\left( \sin^{N-2} \theta\ v(r,\theta)\right)\right).
\end{align*}
Let
\begin{align*}
 U:=\{(r,\theta)\subset(0,\infty)\times(0,\pi): r\notin\overline{\cO}\}.
\end{align*}
Since $u$ solves~\eqref{eq:limit_problem}, then, by Lemma~\ref{lem:regularity_limit}, $v$ is a classical solution of
\begin{align*}
 L v +(p-1)|u|^{p-2}v=0\quad \text{in \ } U,\qquad
 v\leq 0\quad \text{on \ }\partial U,\qquad \lim_{r\to \infty}v(r,\cdot)=0.
\end{align*}
We claim that $v\leq 0$ in $U$. By contradiction, assume that $v(r_0,\theta_0)=\max_{\overline{U}} v>0$ for some $(r_0,\theta_0)\in U$. Then, $\partial_{rr}u(r_0,\theta_0)\leq 0$, \ $\partial_{r}u(r_0,\theta_0)=\partial_{\theta}u(r_0,\theta_0)=0$, \ $\partial_{\theta \theta}u(r_0,\theta_0)\leq 0,$ and
\begin{align*}
 0&=L v(r_0,\theta_0) +(p-1)|u(r_0,\theta_0)|^{p-2}v(r_0,\theta_0)\\
 &= -\partial_{rr}v(r_0,\theta_0)
 -\partial_\theta\left(\frac{1}{r^2 \sin^{N-2}\theta}\partial_{\theta}\left( \sin^{N-2} \theta\ v(r_0,\theta_0)\right)\right)+(p-1)|u(r_0,\theta_0)|^{p-2}v(r_0,\theta_0)\\
  &= -\partial_{rr}v(r_0,\theta_0)
 -
 \frac{\partial_{\theta \theta}v(r_0,\theta_0)}{r^2}+\frac{(N-2) \csc ^2(\theta_0)}{r^2}v(r_0,\theta_0)
 +(p-1)|u(r_0,\theta_0)|^{p-2}v(r_0,\theta_0)>0,
\end{align*}
where we used that
\begin{align*}
 -\partial_\theta\left(\frac{1}{\sin^{N-2}\theta}\right)\partial_\theta (\sin^{N-2}\theta)
 -\frac{1}{\sin^{N-2}\theta}\partial_{\theta\theta} (\sin^{N-2}\theta)
 =(N-2) \csc ^2(\theta)>0\quad \text{ for }\theta\in(0,\pi).
\end{align*}
We have reached a contradiction. This shows that $v\leq 0$ in $U$. Hence, $\theta\mapsto u(r,\theta)$ is also nonincreasing for every $r\in[0,\infty)\smallsetminus\overline{\cO}$. Therefore, $u$ is foliated Schwarz symmetric in $\rn$, as claimed.
\end{proof}

\begin{remark}
\emph{We note that the proof of radial symmetry and of foliated Schwarz symmetry can be easily adapted to yield the same symmetry results for the least energy positive and nodal solutions of~\eqref{eq:Q_eps_problem}.}
\end{remark}

\section{Decay at infinity}\label{d:sec}

In this section we prove Theorem~\ref{thm:decay}. We split the proof into four lemmas.

\begin{lemma}\label{lem:upper_bound_positive}
Let $w$ be a positive solution of~\eqref{eq:limit_problem}. Then there exists $C>0$ such that
\begin{align*}
0<w(x)\leq C |x|^{2-N}\qquad \text{for all \ }x\in \rn.
\end{align*}
\end{lemma}

\begin{proof}
Let $\what\o$ be an open bounded subset of $\rn$ that contains $0$ and $\o$. Let $F(x)=|x|^{2-N}$ for $x\neq 0$.  This is (a multiple of) the fundamental solution for the Laplacian in $\rn$, in particular, $-\Delta F = 0$ in $\rn\smallsetminus\{0\}$. By Lemma~\ref{lem:regularity_limit}, there exists $C>0$ such that $CF-w\geq 0$ on $\partial\what\Omega$. Since $-\Delta (CF-w) = -Q(x) w^{p-1}=w^{p-1}\geq 0$ in $\rn\smallsetminus\what\Omega$, the maximum principle yields $w\leq CF$ in $\rn\smallsetminus\what\Omega$. As $\what\o$ is bounded the result follows.
\end{proof}

\begin{lemma}\label{lem:upper_bound}
Let $w$ be a positive solution of~\eqref{eq:limit_problem} and $u$ be an arbitrary solution of~\eqref{eq:limit_problem}. Then, there exists $C>0$ such that
\begin{align*}
|u(x)|<Cw(x)\qquad \text{for all \ }x\in\rn.
\end{align*}
\end{lemma}

\begin{proof}
Since $w>0$ in $\rn$, there exists $C\geq 1$ such that $Cw-u>0$ on $\overline\Omega$ (see Lemma \ref{lem:regularity_limit}). Let $v:=Cw-u$, then
\begin{align*}
-\Delta v 
=Q(x)(C w^{p-1} - |u|^{p-2}u)
=Q(x)(C^{2-p} (Cw)^{p-1} - |u|^{p-2}u)
=Q(x)[(C^{2-p}-1)(Cw)^{p-1}+c(x)v]
\end{align*}
where
\begin{align*}
c(x):=
\begin{cases}
\frac{(Cw(x))^{p-1} - |u(x)|^{p-2}u(x)}{Cw(x)-u(x)} & \text{ if }Cw(x)\neq u(x),\\
0 & \text{ if }Cw(x)=u(x).
\end{cases}
\end{align*}
Then, for $x\in \rn\smallsetminus\Omega,$
\begin{align*}
-\Delta v +c(x)v
=(1-C^{2-p})(Cw)^{p-1}.
\end{align*}
Since $C>1$ and $p>2$, we have that $1-C^{2-p}>0$.  Furthermore, $c\geq 0$ in $\rn$.  Then, by the maximum principle, we have that $Cw\geq u$ in $\rn\smallsetminus\Omega$. Hence, $u^+\leq Cw$ in $\rn$, where $u^+:=\max\{u,0\}$. The same argument applied to $-u$ shows that $-u^-=(-u)^+\leq Cw$ in $\rn$ for some $C\geq 1$. Hence, $|u(x)|=u^+(x)-u^-(x)<2Cw(x)$ for all $x\in\rn$.
\end{proof}

\begin{lemma}\label{lem:lower_bound_ball}
Let $\vr>0$ and let $u$ be a positive least energy solution of~\eqref{eq:limit_problem} in $\Omega=B_\vr$ for $p\in (\frac{2N-2}{N-2},\frac{2N}{N-2})$. Then, given $\delta>0$, there exists $C_\delta>0$ such that
\begin{equation*}
u(x)>C_\delta|x|^{2-N-\delta}\qquad\text{in \ }\rn\smallsetminus B_1.
\end{equation*}
\end{lemma}

\begin{proof}
By Lemma~\ref{lem:radially}, $u$ is radially symmetric. For $\alpha\in(0,u(\vr)\vr^{N-2})$ let $v_\alpha:\rn\smallsetminus\{0\}\to \r$ be given by $v_{\alpha}(r):=\alpha r^{2-N}$, where $r=|x|$. Then $u> v_\alpha$ on $\partial B_\vr$.  Let $w_\alpha:=u-v_\alpha$. If $w_\alpha>0$ in $\rn\smallsetminus B_\vr$, the claim follows.  Assume then that there is $R_\alpha>\vr$ so that $w_\alpha(R_\alpha)=0$ and $w_\alpha>0$ in $(\vr,R_\alpha)$.  Since $-\Delta (-w_\alpha)=u^{p-1}>0$ in $(R_\alpha,\infty)$, the maximum principle yields that $w_\alpha<0$ in $(R_\alpha,\infty)$. In other words, for every $\alpha\in(0,u(\vr)\vr^{N-2})$, there is $R_\alpha>\vr$ such that
\begin{align}\label{ueps}
    u(r)<\alpha r^{2-N}\qquad \text{for every \ }r>R_\alpha.
\end{align}
Note that $R_\alpha$ is decreasing in $\alpha$. Furthermore,
\begin{align}\label{Reps}
R_\alpha\to \infty\qquad \text{as \ }\alpha\to 0.
\end{align}
Indeed, if $R_\alpha\to R_*$ as $\alpha\to 0$, this would imply that $u(R_*)=\lim_{\alpha\to 0}u(R_\alpha)\leq \lim_{\alpha\to 0} \alpha R_\alpha^{2-N}=0$, which contradicts the positivity of $u.$

Since $p>\frac{2N-2}{N-2}$, we may assume that the given $\delta>0$ satisfies 
\begin{align}\label{pdelta}
p>\frac{2N-2+\delta}{N-2}    
\end{align}
and, as $(N-2)(p-1)-N>0$, we can use~\eqref{Reps} to find $\alpha\in(0,u(\vr)\vr^{N-2})$ such that
\begin{align}\label{eps}
\delta R_\alpha^{(N-2)(p-1)-N}>1.
\end{align}
For this $\alpha$, let $z:\rn\smallsetminus\{0\}\to \r$ be given by $z(r):= \alpha R_\alpha^{\delta}r^{2-N-\delta}$. Then, by the definition of $R_\alpha$,
\begin{align*}
u(R_\alpha)=\alpha R_\alpha^{2-N}=\alpha R_\alpha^{\delta} R_\alpha^{2-N-\delta}=z(R_\alpha).
\end{align*}
A direct computation shows that
\begin{align*}
\Delta z(r) = \alpha R_\alpha^{\delta}(\delta (N-2 + \delta)) r^{-N-\delta}  \qquad \text{ for }r>0.
\end{align*}
Then, using~\eqref{ueps}, for $r>R_\alpha$  we obtain
\begin{align*}
-\Delta (u-z)(r)
&=-u^{p-1}(r) + \alpha R_\alpha^{\delta} r^{-N-\delta}(\delta (N-2 + \delta))\\
&\geq -(\alpha r^{(2-N)})^{p-1} + \alpha \delta R_\alpha^{\delta} r^{-N-\delta}\\
&=\alpha r^{(p-1)(2-N)}(-\alpha^{p-2}  +  \delta R_\alpha^{\delta} r^{-N-\delta+(p-1)(N-2)}).
\end{align*}
Observe that $-N-\delta+(p-1)(N-2)>0$, by~\eqref{pdelta}. Then, by~\eqref{eps},
\begin{align*}
-\Delta (u-z)(r)
\geq \alpha r^{(p-1)(2-N)}\left(-1+\delta R_\alpha^{-N+(p-1)(N-2)}\right)>0\quad \text{ for }r>R_\alpha.
\end{align*}
By the maximum principle, we have that $u(r)>z(r)=\alpha R_\alpha^{\delta}r^{2-N-\delta}$ for $r>R_\alpha$.  Since $B_{R_\alpha}$ is a compact set, we can find $C>0$ such that $u(x)>C|x|^{2-N-\delta}$ for $|x|>1,$ as claimed.
\end{proof}

\begin{remark} \label{rem:lower bound}
\emph{One cannot expect Lemma~\ref{lem:lower_bound_ball} to be true for every $p\in(2,2^*)$. Indeed, if $\beta\in(N-2,\frac{N}{p})$ and $w\in L^p(\rn)$ is a positive radial function, nonincreasing in the radial direction, then there exists $R>0$ such that
$$w(x)\leq|x|^{-\beta}\qquad\text{if \ }|x|\geq R.$$
Otherwise, there would exist a sequence of positive numbers $r_n\to\infty$ such that $w(x)\geq r_n^{-\beta}$ if $|x|\leq r_n$. Then, as $\beta p<N$, we would have that
$$\infty>\irn|w|^p\geq \int_{|x|\leq r_n}|w|^p\geq \int_{|x|\leq r_n}r_n^{-\beta p}=|B_1|\,r_n^{N-\beta p}\to\infty.$$
This is impossible. Note however that, as $p>2$, the interval $(N-2,\frac{N}{p})$ can only be nonempty if $N=3$.}
\end{remark}

\begin{lemma}\label{lem:domain_comparison}
Let $\Omega\subset B_\vr$ for some $\vr>0$, $w$ be a positive solution of~\eqref{eq:limit_problem} in $\o$ and $u$ be a positive solution of~\eqref{eq:limit_problem} in $B_\vr$. Then there is $C>0$ such that
\begin{align*}
w(x)>C u(x)\qquad \text{for all \ }x\in\rn.
\end{align*}
\end{lemma}

\begin{proof}
Let $C\in(0,1)$ be such $v:=w - Cu>0$ in $\overline{B}_\vr$. Then, for $|x|>\rho$,
\begin{align*}
-\Delta v 
=Cu^{p-1}-w^{p-1}
=((Cu)^{p-1}-w^{p-1})+(1-C^{p-2})Cu^{p-1}
=-c(x)v+(1-C^{p-2})u^{p-1},
\end{align*}
where
\begin{align*}
c(x):=
\begin{cases}
\frac{w^{p-1}(x) - (Cu(x))^{p-1}}{w(x)-Cu(x)} &\text{if \ }w(x)\neq Cu(x),\\
0 &\text{if \ }w(x)=Cu(x).
\end{cases}
\end{align*}
That is, for $|x|>\rho,$
\begin{align*}
-\Delta v +c(x)v
=(1-C^{p-2})u^{p-1}.
\end{align*}
Since $C\in(0,1)$ and $p>2$, we have that $1-C^{p-2}>0$.  Furthermore, $c\geq 0$ in $\rn.$  Then, by the maximum principle, we have that $v\geq 0$ in $\rn\smallsetminus B_\vr$ and the claim follows.
\end{proof}
\medskip

\begin{proof}[Proof of Theorem~\ref{thm:decay}] Statement $(i)$ follows from Lemmas~\ref{lem:upper_bound_positive} and~\ref{lem:upper_bound}, and statement $(ii)$ follows from Lemmas~\ref{lem:lower_bound_ball} and~\ref{lem:domain_comparison}.
\end{proof}

\begin{remark}\label{exp:rmk}
\emph{In contrast, every solution $u$ of~\eqref{eq:Q_problem} satisfies
\begin{align*}
 |u(x)|<Ce^{-\eps |x|}\qquad \text{for every \ }x\in \rn
\end{align*}
and some constant $C=C(\eps,u)>0$. This fact is proved using standard arguments (using an exponential function as supersolution for comparison).}
\end{remark}

\bigskip

\begin{flushleft}
\textbf{Mónica Clapp}\\
Instituto de Matemáticas\\
Universidad Nacional Autónoma de México \\
Campus Juriquilla\\
76230 Querétaro, Qro., Mexico\\
\texttt{monica.clapp@im.unam.mx} 

\bigskip 

\textbf{Víctor Hernández-Santamaría} and \textbf{Alberto Saldaña}\\
Instituto de Matemáticas\\
Universidad Nacional Autónoma de México \\
Circuito Exterior, Ciudad Universitaria\\
04510 Coyoacán, Ciudad de México, Mexico\\
\texttt{victor.santamaria@im.unam.mx, \ alberto.saldana@im.unam.mx}
\end{flushleft}

\end{document}